\documentclass[a4paper,12pt, reqno]{amsart}
\usepackage{amsthm, amsmath, amssymb, amsfonts, amsopn, color}
\usepackage{enumerate}

\newtheorem{theorem}{Theorem}[section]

\newtheorem{proposition}[theorem]{Proposition}
\newtheorem{lemma}[theorem]{Lemma}
\numberwithin{equation}{section}
\newcommand{\LI}{L^\infty(\mathbb{R}^2)}
\newcommand{\RN}{\mathbb{R}^2}

\newcommand{\e}{\varepsilon}

\numberwithin{equation}{section}

\begin{document}
\title{Non-topological solutions in a
generalized  Chern-Simons model on torus}
\author{Youngae Lee}
\address[Youngae Lee]{Center for Advanced Study in Theoretical Science, National Taiwan University, No.1, Sec. 4, Roosevelt Road, Taipei 106, Taiwan}
\email{youngaelee0531@gmail.com}

\begin{abstract}
We consider a   quasi-linear elliptic equation with Dirac source terms arising in a
generalized self-dual Chern-Simons-Higgs gauge theory.
In this paper, we study doubly periodic vortices with arbitrary vortex configuration.
First of all,  we show that under doubly periodic condition, there are only two types of solutions, topological and non-topological
solutions as   the coupling parameter goes to zero.
Moreover, we succeed to construct  non-topological solution with $k$ bubbles where $k\in\mathbb{N}$ is any given number.
We believe that it is the first  result for the existence of non-topological doubly periodic solution of the   quasi-linear elliptic equation arising in a generalized self-dual Chern-Simons-Higgs gauge theory.
To find a suitable approximate solution, it is important to understand  the structure of  quasi-linear elliptic equation.
\end{abstract}

\date{\today}

\maketitle

\section{Introduction}
In this paper, we study  a  generalized self-dual Chern-Simons-Higgs gauge theory introduced  by  Burzlaff,    Chakrabarti,   Tchrakian  in \cite{BCT}. The Lagrangian density of the model in $(2+1)$ dimensions is
\[\mathcal{L}=\sqrt{2}\e\epsilon^{\mu\nu\alpha}\Big[A_\alpha-2i\Big(1-\frac{|\phi|^2}{2}\Big)\phi\overline{D_\mu\phi}\Big]F_{\mu\nu}+2(1-|\phi|^2)^2|D_{\mu}\phi|^2-V,\]
where $A=(A_0,A_1,A_2)$ is a 3-vector  gauge field, $F_{\alpha\beta}=\frac{\partial}{\partial x_\alpha} A_\beta-\frac{\partial}{\partial x_\beta}A_\alpha$ is the corresponding curvature, $\phi=\phi_1+i\phi_2$ is a complex scalar field called the Higgs field, $D_j=\frac{\partial}{\partial x_{j}}-iA_{j}$, $j=0,1,2$ is the gauge covariant derivative associated with $A$, $\alpha, \beta, \mu, \nu=0, 1, 2,$ $\e>0$ is a constant referred to as the Chern-Simons coupling parameter, $\epsilon^{\alpha\beta\gamma}$ is the Levi-Civita totally skew-symmetric tensor with $\epsilon^{012}=1$, $V$ is the Higgs potential function.
The corresponding Bogomol'nyi equations for unknowns $\phi$, $A$ defined on $\RN$ are \begin{equation*}
\begin{aligned}
\left\{
\begin{array}{ll}
& D_1\phi=iD_2\phi,
\\& (1-|\phi|^2)F_{12}=i(D_1\phi\overline{D_2\phi}-\overline{D_1\phi}D_2\phi)+\frac{1}{2\e^2}|\phi|^2(1-|\phi|^2)^2.
 \end{array}\right.\end{aligned}
\end{equation*}
In view of  Jaffe-Taubes' argument in \cite{JT}, we introduce unknown $v$ defined by
\[\phi(z)=\exp\Big(\frac{v(x)}{2}+i\sum_{j=1}^N\mbox{arg}(z-p_j)\Big), \ z=x_1+ix_2\in\mathbb{C},\]where $\{p_j\}_{j=1}^N$ are the zeros of $\phi(z)$, allowing their multiplicities.
Then we obtain the following reduced equation:
\begin{equation}
(1-e^{v})\Delta v-e^{v(x)}|\nabla v|^2 +\frac{1}{\e^{2}}e^{v(x)}\left(  1-e^{v(x)}\right)^2  =4\pi%
%TCIMACRO{\dsum \limits_{i=1}^{k_2}}%
%BeginExpansion
{\displaystyle \sum \limits_{j=1}^{N}}
%EndExpansion
\delta_{p_j}.\label{01}
\end{equation}
Here $p_j$ is called a vortex point.   The equation \eqref{01} can be considered in $\RN$ or a two dimensional
flat torus $\Omega$  due to the theory suggested by 't Hooft in \cite{'tH}.

We fix $\e>0$ for a while. In $\RN$,  a solution $v(x)$ is called a topological solution if $\lim_{|x|\to+\infty}v(x)=0$, and is called a non-topological solution if   $\lim_{|x|\to+\infty}v(x)=-\infty$.
Yang in \cite{Yang} found topological multi-vortex solutions of \eqref{01} by using the variational structure of the elliptic problem to produce an iteration scheme that
yields the desired solution. After then, Chae and Imanuvilov in \cite{CI} constructed a  non-topological multi-vortex solution $v(x)$  of \eqref{01} satisfying $v(x)=-(2N+4+\sigma)\ln|x|+O(1)$ as $|x|\to+\infty$ for some $\sigma>0$. To obtain the non-topological solution of \eqref{01}, the authors in \cite{CI} observed that \eqref{01} is a perturbation of the Liouville equation
 and applied    the arguments developed in \cite{CI0}. In \cite{CI0}, Chae and Imanuvilov showed the existence of non-topological multi-vortex solutions
  of   the relativistic Chern-Simons-Higgs  model (see \eqref{sceq} below), using the   implicit function theorem argument with Lyapunov-Schmidt reduction method.

Now we consider the equation \eqref{01} on flat two torus $\Omega$, where $\e$ goes to $0$. Since $(1-e^{v})\Delta v-e^{v(x)}|\nabla v|^2 =\mbox{div}((1-e^{v})\nabla v)$, any solution $v(x)$ to \eqref{01} satisfies \begin{equation}\label{uniforml1}\int_{\Omega}\frac{1}{\e^2}e^{v(x)}(1-e^{v(x)})^2dx=4\pi N.\end{equation}
Moreover, from the maximum principle (see also \cite[Lemma 3.1]{H}), we note that any solution $v(x)$ to \eqref{01} satisfies
\begin{equation}\label{negative}v(x)\le0\ \  \textrm{on}\ \  \Omega.\end{equation}
For the well known Chern-Simons-Higgs equation with $\e\to0$ (see \eqref{sceq} below), the corresponding properties \eqref{uniforml1} and \eqref{negative} were important to classify the
solutions according to their asymptotic behavior as $\e\to0$ (see \cite{DJLPW,CK}).
So it is natural that we expect that the solutions to \eqref{01} can also be classified according to the asymptotic behavior.
Now we have the following theorem:
\begin{theorem}\label{Lp}
For any given vortex configuration $\{p_j\}$,  let $v_{\e}$ be a sequence of solutions of (\ref{01}).  Then, up to  subsequence,
one of the following holds true:

(i)   $v_{\e}\to0$ a.e. as $\varepsilon\to0$. Moreover, $v_{\e}\to0$  in $L^p(\Omega)$ for any $p>1$ (topological type);

(ii)  $v_{\e}\to-\infty$ a.e. as $\varepsilon\to0$ (non-topological type).
\end{theorem}
Recently, Han in \cite[Theorem 3.1]{H} proved the existence of critical value of the coupling parameter $\e_c=\e_c(p_1,...,p_N)>0$ such that there is a solution to
\eqref{01} on $\Omega$ if and only if $0<\e\le\e_c$. He obtained a maximal solution $v_{\e,M}$ to \eqref{01} by using a super-sub solutions method (see \cite[Theorem 2.1]{H}).
 Here the maximal solution means that $v_{\e, M}\ge v_\e$  on $\Omega$ for any solution $v_\e$ to \eqref{01}.  In \cite[Lemma 3.5]{H}, he also showed that
  the maximum solutions $v_{\e,M}$ of \eqref{01} are a monotone family in the sense that $v_{\e_1,M}>v_{\e_2,M}$ whenever
$0<\e_1 < \e_2 < \e_c$. Therefore, in view of Theorem \ref{Lp},  the maximal solution obtained in \cite{H} is a topological solution.

At this point, one might ask the existence of non-topological solution to \eqref{01} on $\Omega$.
In this paper, we obtain the affirmative answer for this question by  constructing a bubbling non-topological solution solution $v_\e$ to \eqref{01} on $\Omega$ satisfying
\begin{equation}\label{finalgoal}
\lim_{\e\to0} \sup_{\Omega}v_{\e}=-\infty,\
\frac{e^{v_{\e}}}{\int_{\Omega}e^{v_{\e}}dx
}\rightarrow \frac{1}{k}\sum_{i=1}^{k}\delta_{q_{i}}, \ q_i\in\Omega\setminus[\cup_{i=1}^{2k}\{p_{i}\}],
\end{equation}  in the sense of
measure as  $\e\to0$.\\
 For the construction of  bubbling  solution solution $v_\e$ to \eqref{01} on $\Omega$ satisfying \eqref{finalgoal},  we assume that $N=2k\in2\mathbb{N}$. We note that since the equation \eqref{01} is quasi-linear, it is not easy to deal it directly.
As in \cite{H,TY,Yang}, we introduce a new dependent variable $u$ defined by
\[u=F(v):=1+v-e^v.\]
 We have that $F'(v)=1-e^v$ and $F''(v)=-e^v$, which implies $F$ is strictly increasing and invertible over $(-\infty,0)$. Let $G$ be the inverse function of $F$ over $(-\infty,0]$.
Then we see that $G(u)=v=G(1+v-e^v)$. Let $u_\e=F(v_\e)=1+v_\e-e^{v_\e}$. Then
$v_\e$ satisfies \eqref{01} if and only if $u_\e$ satisfies
\begin{equation}
\Delta u_\e+\frac{1}{\varepsilon^{2}}e^{G(u_\e(x))}\left(  1-e^{G(u_\e(x))}\right)^2=4\pi%
%TCIMACRO{\dsum \limits_{i=1}^{k_2}}%
%BeginExpansion
{\displaystyle \sum \limits_{i=1}^{2k}}
%EndExpansion
\delta_{p_i}.
 \label{001}%
\end{equation}  We remark that if $\lim_{\e\to0} \sup_{\Omega}u_{\e}=-\infty$, then the equation  \eqref{001} would be a perturbation of  bubbling solutions $W_\e$ of the following Chern-Simons-Higgs equation:
\begin{equation}
\Delta W_\e+\frac{1}{\varepsilon^{2}}e^{W_\e(x)}\left(  1-e^{W_\e(x)}\right)  =4\pi%
%TCIMACRO{\dsum \limits_{i=1}^{k_2}}%
%BeginExpansion
{\displaystyle \sum \limits_{i=1}^{2k}}
%EndExpansion
\delta_{p_i}\ \ \textrm{on}\ \ \Omega.  \label{sceq}%
\end{equation}
The  relativistic Chern-Simons-Higgs model has been proposed  in \cite{HKP} and independently in \cite{JW} to describe vortices in high temperature
superconductivity.  The above equation   was derived from the
Euler-Lagrange equations of the CSH model via a vortex ansatz, see \cite{HKP, JW, T2, Y}.
The equation \eqref{sceq} has been extensively studied not only in a flat torus $\Omega$ but also   in the whole $\RN$.
 We refer the readers to \cite{CY,CI0, CFL,  Choe0, Choe, CK,LY0, LY1, NT, NT1, SY0, T1,T2} and references therein.
 Among them, in a recent paper \cite{LY1}, Lin and Yan  succeeded to construct  bubbling non-topological solutions   to \eqref{sceq} on $\Omega$.
 Compared to \eqref{sceq}, our equation has a difficulty caused by the nonlinear terms including implicit function $G$.   Therefore, to choose a suitable approximate solution, we should investigate the behavior of the function $G$ near $-\infty$  and carry out the analysis carefully.
To state our  result exactly, we introduce the following notations:\\
Let $G$ be the Green function satisfying
\begin{equation*}
-\Delta_x G(x,y)=\delta_y -\frac{1}{|\Omega|} \quad\mbox{for }~ x, y\in \Omega,
 \quad\mbox{and}\quad \int_\Omega G(x,y)dx=0.
\end{equation*}
We let $\gamma(x,y)=G(x,y)+\frac{1}{2\pi}\ln|x-y|$ be the regular part of the Green function $G(x,y)$, and
\begin{equation*}
  u_0(x) \equiv
 -4\pi\sum^{N}_{i=1}G(x,p_{i}).
\end{equation*}
Then     $u_0$ satisfies the following problem:%
\begin{equation*}
\begin{aligned}
\left\{
\begin{array}{ll}
&\Delta u_0 =-\frac{4\pi N}{|\Omega|}+4\pi\sum_{i=1}^{N}\delta_{p_i},\\
&\int_{\Omega}u_0dx=0.
 \end{array}\right.\end{aligned}
\end{equation*}We remind that $N=2k$.
We denote $\Omega^{(k)}:=\{(x_1,...,x_k)\ | x_i\in\Omega\setminus[\cup_{i=1}^{2k}\{p_{i}\}]\ \textrm{for}\ 1\le i\le k, x_i\neq x_j \ \textrm{if}\ i\neq j\}$.
Let ${\bf{q}}= \left(  q_{1},...,q_{k}\right)\in\Omega^{(k)}$ be the critical point of the following function:
\[
G^{\ast}\left(  \bf{q}\right)  :=\sum_{i=1}^{k}u_0\left(  q_{i}\right)  +8\pi
\sum_{i\neq j}G\left(  q_{i},q_{j}\right)  .
\]
We define%
\[
D\left(  \bf{q}\right)  :=\lim_{r\rightarrow0}\left(  \sum_{i=1}^{k}\rho_{i}\left(
\int_{\Omega_{i}\setminus B_{r}\left(  q_{i}\right)  }\frac{e^{f_{{\bf{q}},i}}%
-1}{\left \vert y-q_{i}\right \vert ^{4}}dy-\int_{\mathbb{R}^2\setminus \Omega_{i}}\frac
{1}{\left \vert y-q_{i}\right \vert ^{4}}dy\right)  \right)  ,
\]
where $\Omega_{i}$ is any open set satisfying $\Omega_{i}\cap \Omega
_{j}=\emptyset$ if $i\neq j$, $\cup_{i=1}^{k}\bar{\Omega}_{i}=\bar{\Omega}$,
$B_{d_{i}}\left(  x_{i}\right)  \subset\subset \Omega_{i}$, $i=1,...,k,$%
\[
f_{{\bf{q}},i}\left(  y\right)  :=8\pi \left(  \gamma \left(  y,q_{i}\right)
-\gamma \left(  q_{i},q_{i}\right)  +\sum_{j\neq i}\left(  G\left(
y,q_{j}\right)  -G\left(  q_{i},q_{j}\right)  \right)  \right)  +u_0\left(
y\right)  -u_0\left(  q_{i}\right),
\]
and
\[
\rho_{i}=e^{8\pi \left(  \gamma \left(  q_{i},q_{i}\right)  +\sum_{j\neq
i}G\left(  q_{i},q_{j}\right)  \right)  +u_0\left(  q_{i}\right)  }.
\]
At this point, we introduce our main result.
\begin{theorem}\label{blmix}
Let ${\bf{q}}=(q_1,...,q_k)\in\Omega^{(k)}$ be a non-degenerate critical point of $G^{\ast}\left(  \bf{q} \right)  $.
Suppose that $D\left(  \bf{q}\right)  <0$. Then for $\varepsilon>0$ small, there
exists a non-topological solution solution $v_{\varepsilon}$  to \eqref{01} such that%
\[\lim_{\e\to0} \sup_{\Omega}v_{\e}=-\infty,\ \
\frac{e^{v_{\e}}}{\int_{\Omega}e^{v_{\e}}dx
}\rightarrow \frac{1}{k}\sum_{i=1}^{k}\delta_{q_{i}}\text{ in the sense of
measure as}\ \e\to0.
\]
\end{theorem}To the best of our knowledge,   Theorem \ref{blmix} is the first result for the existence of    non-topological solution solutions to \eqref{01}  on $\Omega$.
We  remark that in our paper, a  limiting equation for \eqref{01} is Liouville equation since $\lim_{\e\to0} \sup_{\Omega}v_{\e}=-\infty$. It would be an interesting  problem to find other types of  non-topological solution solution  to \eqref{1}, for example,  satisfying  $\sup_\Omega v_{\e}\ge -c_0>-\infty$ for some constant $c_0>0$.

The organization of this paper is as follows.
In Section 2, we prove Theorem \ref{Lp}. In Section 3, to prove Theorem \ref{blmix},   we present some preliminaries results and discuss about the invertibility of a  linearized operator.
Moreover, we find a suitable approximate solution and complete the proof of Theorem \ref{blmix}.

\section{proof of Theorem \ref{Lp}}

\emph{Proof of Theorem \ref{Lp}:}\\ Our arguments will be based on \cite[Theorem 3.1]{DJLPW}. We consider the equation \eqref{001}, which is equivalent to \eqref{01}. Let $\{u_\e\}$ be a sequence of solutions of \eqref{001}.
Let $d_{\varepsilon}=\frac{1}{|\Omega|}\int_{\Omega} u_{\e} dx$ and $u_{\e}=w_{\varepsilon}+u_{0} +d_{\varepsilon}$. Then
 $w_\e$ satisfies
\begin{equation}
\begin{aligned}\label{wemain}
 \Delta w_{\e}+\frac{1}{\varepsilon^2} e^{G(u_{\e})}(1-e^{G(u_{\e})})^2=\frac{4\pi N}{|\Omega|}\quad\mbox{on }~ \Omega,
\end{aligned}
\end{equation}
and  $\int_{\Omega} w_{\varepsilon} dx=0$.\par
We claim that  there exist $C_q>0$ such that $\|\nabla w_{\varepsilon}\|_{L^q(\Omega)}\le C_q$ for any $q\in(1,2)$.
Let $q'=\frac{q}{q-1}>2$. Then
\begin{equation}\begin{aligned}\label{normexpression}
&\|\nabla w_{\varepsilon}\|_{L^q(\Omega)}
\\&\le\sup\Big\{\Big|\int_{\Omega}\nabla w_{\varepsilon}\nabla\phi dx\Big|\  \Big|\ \ \phi\in W^{1,q'}(\Omega),\ \int_{\Omega}\phi dx=0,\ \|\phi\|_{W^{1,q'}(\Omega)}=1\Big\}.
\end{aligned}\end{equation}
By lemma 7.16 in \cite{GT}, if  $\int_{\Omega}\phi dx=0$, then there exist $c,\ C>0$ such that
\begin{equation}\label{GTineq}
|\phi(x)|\le c\int_{\Omega}\frac{|\nabla\phi|}{|x-y|}dy\le C\|\nabla\phi\|_{L^{q'}(\Omega)}\ \ \textrm{for}\ x\in\Omega.
\end{equation}
Thus  in view of  (\ref{wemain}),  (\ref{GTineq}), and  \eqref{uniforml1}, we see that there exists constant $C>0$, independent of $\phi$ satisfying $\int_{\Omega}\phi dx=0$ and $\|\phi\|_{W^{1,q'}(\Omega)}=1$, such that
\begin{equation}
\begin{aligned}
\Big|\int_{\Omega}\nabla w_{\varepsilon}\nabla\phi dx\Big|&=\Big|\int_{\Omega}\Delta w_{\varepsilon}\phi dx\Big|
\\&\le\|\phi\|_{L^\infty(\Omega)}\Big|\int_{\Omega} |\frac{1}{\varepsilon^2} e^{G(u_{\varepsilon})}(1-e^{G(u_{\e})})^2|dx+4\pi N \Big|\le C.
\end{aligned}
\end{equation}
Now using (\ref{normexpression}), we complete the proof of our claim.\par
In view of  Poincar\'{e} inequality, we also have $\|w_{\varepsilon}\|_{L^q(\Omega)}\le c\|\nabla w_{\varepsilon}\|_{L^q(\Omega)}$. Then there exist  $w \in W^{1,q}(\Omega)$ and $p>1$ such that, as $\varepsilon\to0$,
\begin{equation}\label{wecone}
w_{\varepsilon}\rightharpoonup w \ \ \textrm{weakly in}\  W^{1,q}(\Omega),\ w_{\varepsilon}\to w \ \ \textrm{strongly in}\  L^p(\Omega),\  w_{\varepsilon}\to w \  \textrm{a.e.}.
\end{equation}
 Since $v_{\e}\le0$ on $\Omega$, we see that  $u_{\e}\le0$ and $0\le e^{d_{\varepsilon}}\le1$. Then there exists $A \ge0$ such that  $\limsup_{\varepsilon\to0}e^{d_{\varepsilon}}=A $.
If $A\equiv0$, that is,  $\lim_{\varepsilon\to0}d_{\varepsilon}=-\infty$, then  by using \eqref{wecone}, we get that $u_{\e}=w_{\varepsilon}+d_{\varepsilon}+u_{0}\to -\infty$ a.e. in $\Omega$.\\
If $A>0$, then by using Fatou's lemma and  (\ref{wecone}),  we see that
\begin{equation*}
\begin{aligned}
4\pi N \varepsilon^2&=
\int_{\Omega}e^{G(u_{\varepsilon})}(1-e^{G(u_{\e})})^2dx
\ge\int_{\Omega}  e^{G(w+u_0+\ln A)}(1-  e^{G(w +u_{0}+\ln A)})^2dx,
\end{aligned}
\end{equation*}
which implies that $G(w+u_0+\ln A)=-\infty$ or $G(w+u_0+\ln A)=0$  a.e. in $\Omega$.
Since $G$ is strictly increasing on $(-\infty,0)$ and $G(0)=0$, we see that $w+u_0+\ln A=-\infty$ or $w+u_0+\ln A=0$ a.e. in $\Omega$.
By $A>0$ and $w, u_0\in L^p(\Omega)$, we have $w+u_0+\ln A=0$ a.e. in $\Omega$.  From $\int_{\Omega} w+u_0dx=0$, we see that $A\equiv 1$, and $w+u_0=0$ a.e. in $\Omega$.
By using \eqref{wecone}, we get that  $u_{\e}=w_{\varepsilon}+d_{\varepsilon}+u_{0}\to w +\ln A +u_{0}=0$ a.e. in $\Omega$ and $u_{\e}\to0$ in $L^p(\Omega)$ for any $p>1$ (since $q\in(1,2)$ in \eqref{normexpression} can be arbitrary). Now we complete the proof of Theorem \ref{Lp}.
\hfill$\Box$

\section{Existence of bubbling non-topological solution   solution}
In this section,  we want to construct a bubbling non-topological solution solution $v_\e$ to \eqref{01} satisfying $\lim_{\e\to0} \sup_{\Omega}v_{\e}=-\infty$,
and \[
\frac{e^{v_{\e}}}{\int_{\Omega}e^{v_{\e}}dx
}\rightarrow \frac{1}{k}\sum_{i=1}^{k}\delta_{q_{i}}\text{ in the sense of
measure as}\ \e\to0,
\]
where ${\bf{q}}=(q_1,...,q_k)\in\Omega^{(k)}$ is a non-degenerate critical point of $G^{\ast}\left(  \bf{q} \right)  $
and $D\left(  \bf{q}\right)  <0$.\\
Without loss of generality, from now on,  we assume that  $|\Omega|=1$.

We note that  $v_\e$ satisfies \eqref{01} if and only if $u_\e=F(v_\e)=1+v_\e-e^{v_\e}$ satisfies
\begin{equation}
\Delta u_\e+\frac{1}{\varepsilon^{2}}e^{G(u_\e(x))}\left(  1-e^{G(u_\e(x))}\right)^2  =4\pi%
%TCIMACRO{\dsum \limits_{i=1}^{k_2}}%
%BeginExpansion
{\displaystyle \sum \limits_{i=1}^{2k}}
%EndExpansion
\delta_{p_i}.
\label{1}%
\end{equation}

As we mentioned in the introduction,  if  $\lim_{\e\to0}\sup_{\Omega}u_{\varepsilon}=-\infty$, then  $u_{\varepsilon}$ would be related to the following
Chern-Simons-Higgs equation:%
\[
\Delta W_\e+\frac{1}{{\varepsilon}^{2}}e^{W_\e
(y)}\left(  1-e^{W_\e(y)}\right)  =4\pi\sum_{j=1}^{2k}  \delta_{p_{j}}.
\]
  In \cite{LY1}, bubbling solutions for the above Chern-Simons-Higgs equation have been
constructed as following:%
\[%
\begin{array}
[c]{ccccc}%
W_{\varepsilon}(y) & \simeq & u_0(y)+w_{\bf{x},\mu}^{\ast}(y)-\int_{\Omega}w_{\bf{x},\mu}%
^{\ast}(z)dz  +c\left(  w_{\bf{x},\mu}\right),%
\end{array}
\]
where ${\bf{x}}= (  x_{1},...,x_{k} ),$ $x_i\in\Omega$,  $\mu\in \lbrack \frac{\beta_{0}}{\sqrt{{\varepsilon}}},\frac{\beta_{1}}%
{\sqrt{{\varepsilon}}}]\ \textrm{ for some}\  0<\beta_{0}\ll1,\ \beta_{1}\gg1,$
\[
\rho_{i}:=e^{8\pi \gamma \left(  x_{i},x_{i}\right)  +8\pi
\sum_{j\neq i}G\left(  x_{j},x_{i}\right)  +u_0\left(  x_{i}\right)  },\]
\[ \left(  \mu_{1},...,\mu
_{k}\right):=\Big(\mu,\sqrt{\frac{\rho_{1}}{\rho_{2}}}\mu,...,\sqrt{\frac{\rho_{1}}{\rho_{k}}}\mu\Big),
\]
and $d>0$ is a fixed small constant,   $d_{i}^{2}:=d-1/\mu_{i}^{2},\ u_{x_{i},\mu_{i}}(y):=\ln \frac{8\mu_{i}^{2}}{\left(  1+\mu_{i}^{2}\left \vert
y-x_{i}\right \vert ^{2}\right)  ^{2}},$
\[%
\begin{array}
[c]{rcl}%
w_{\bf{x},\mu}^{\ast}\left(  y\right)  & := & \sum_{i=1}^{k}w_{x_{i},\mu_{i}}%
^{\ast}\left(  y\right)  ,\\
w_{x_{i},\mu_{i}}^{\ast}\left(  y\right)  & := & \left \{
\begin{array}
[c]{ll}%
u_{x_{i},\mu_{i}}\left(  y\right)  +8\pi \gamma \left(  y,x_{i}\right)  \left(
1-\frac{1}{{d} \mu_{i}^{2}}\right)  , & y\in B_{d_{i}}\left(  x_{i}\right)
,\\
u_{0,\mu_{i}}\left(  d_{i}\right)  +8\pi \left(  G\left(  y,x_{i}\right)
-\frac{1}{2\pi}\ln \frac{1}{d_{i}}\right)  \left(  1-\frac{1}{{d} \mu_{i}%
^{2}}\right)  , & y\in \Omega\setminus B_{d_{i}}\left(  x_{i}\right)  ,
\end{array}
\right.
\end{array}
\text{ }%
\] \[w_{\bf{x},\mu}(y):=w_{\bf{x},\mu}^{\ast}(y)-\int_{\Omega}w_{\bf{x},\mu}%
^{\ast}(z)dz,\]
  \[
 c\left(  w_{\bf{x},\mu}\right):=\ln\frac{16k\pi {\varepsilon}^{2}}{\int_{\Omega
}e^{u_0+w_{\bf{x},\mu}}dy\left(  1+\sqrt{1-32k\pi {\varepsilon}^{2}\frac{\int_{\Omega
}e^{2( u_0+w_{\bf{x},\mu})  }dy}{\left(  \int_{\Omega}e^{u_0%
+w_{\bf{x},\mu}}dy\right)  ^{2}}}\right)  }.
\]
We note that $u_{x_{i},\mu_{i}}$ satisfies
\[
\left \{
\begin{array}
[c]{rcl}%
-\Delta u_{x_{i},\mu_{i}}\left(  y\right)  & = & e^{u_{x_{i},\mu_{i}}\left(
y\right)  }\text{ in }\mathbb{R}^2,\\
\int_{\mathbb{R}^2}e^{u_{x_{i},\mu_{i}}\left(  y\right)  } dy& = & 8\pi.
\end{array}
\right.
\]
 We denote  \[\tilde{W}_{\bf{x},\mu}(y):=w_{\bf{x},\mu}^{\ast}(y)-\int_{\Omega}w_{\bf{x},\mu}^{\ast}(z)dz
+c\left(  w_{\bf{x},\mu}\right).\]
We want to find solution $u_\e$ to \eqref{1} in the following form:
\begin{equation}\label{approximatesol}
\begin{aligned}
u_\e(y)&=1+
u_0(y)+w_{\bf{x},\mu}^{\ast}(y)-\int_{\Omega}w_{\bf{x},\mu}^{\ast}(z)dz
+c\left(  w_{\bf{x},\mu}\right)  +\eta_{{\bf{x}},\mu}(y)\\&=1+u_0+\tilde{W}_{\bf{x},\mu}+\eta_{{\bf{x}},\mu},
\end{aligned}
\end{equation}
where $\eta_{{\bf{x}},\mu}$ is a perturbation term.
To find  $\eta_{{\bf{x}},\mu}$ which makes that $u_\e$ in the form \eqref{approximatesol} is a solution to \eqref{1},  we consider the following linearized operator
\begin{equation*}
\begin{aligned}
 L_{{\bf{x}},\mu}\left(  \eta_{{\bf{x}},\mu}\right)   :=   \left(  \Delta+h_{\mu}\left(  y\right)
\right)  \eta\quad \textrm{with}\
h_{\mu}\left(  y\right)  :=\sum_{i=1}^{k}1_{B_{d_{i}}\left(  x_{i}\right)
}e^{u_{x_{i},\mu_{i}}\left(  y\right)}.
\end{aligned}
\end{equation*}

We see that    $u_\e$ is a solution to \eqref{1} if
$\eta_{{\bf{x}},\mu}$ satisfies
\begin{equation}\label{err}
L_{\bf{x},\mu}\eta_{{\bf{x}},\mu}= g_{{\bf{x}},\mu}\left(
\eta_{{\bf{x}},\mu}\right),
\end{equation}where%
\begin{equation*}
\begin{aligned}
&g_{{\bf{x}},\mu}(\eta_{{\bf{x}},\mu}) :=   h_{\mu}\left(
y\right)  \eta_{{\bf{x}},\mu}+ {8k\pi}-\Delta \tilde{W}_{\bf{x},\mu}-\frac{1}{{\varepsilon}^{2}}e^{G(1+u_0+\tilde{W}_{\bf{x},\mu}+\eta_{{\bf{x}},\mu})}(1-
e^{G(1+u_0+\tilde{W}_{\bf{x},\mu}+\eta_{{\bf{x}},\mu})})^2
.
\end{aligned}
\end{equation*}

To show the invertibility of the linear operator $L_{\bf{x},\mu}$, we need to introduce suitable function spaces.
For fixed a small constant $\alpha\in(0,\frac{1}{2})$, we define
\[
\rho \left(  y\right)  =\left(  1+\left \vert y\right \vert \right)
^{1+\frac{\alpha}{2}},\ \hat{\rho}\left(  y\right)  =\frac{1}{\left(
1+\left \vert y\right \vert \right)  \left(  \ln \left(  2+\left \vert
y\right \vert \right)  \right)  ^{1+\frac{\alpha}{2}}}.
\]
Let $\Omega^{\prime}:=\cup_{i=1}^{k}B_{d_{i}}\left(  x_{i}\right)  $ and $\tilde{\xi}_{i}\left(  y\right)  :=\xi \left(  x_{i}+\mu^{-1}y\right)$. We say that $\xi\in \mathbb{X}_{\alpha,{\bf{x}},\mu}$ if
\begin{align*}\left \Vert \xi \right \Vert
_{\mathbb{X}_{\alpha,{\bf{x}},\mu}}^{2}&:=\sum_{i=1}^{k}\left(  \left \Vert
\Delta \tilde{\xi}_{i}\rho \right \Vert _{L^{2}\left(  B_{2d_{i}\mu_{i}}(0)\right)
}^{2}+\left \Vert \tilde{\xi}_{i}\hat{\rho}\right \Vert _{L^{2}\left(
B_{2d_{i}\mu_{i}}(0)\right)  }^{2}\right) +\left \Vert \Delta \xi \right \Vert _{L^{2}\left(  \Omega \setminus
\Omega^{\prime}\right)  }^{2}+\left \Vert \xi \right \Vert _{L^{2}\left(
\Omega \setminus \Omega^{\prime}\right)  }^{2}\\&<+\infty,\end{align*}
and $\xi\in \mathbb{Y}_{\alpha,{\bf{x}},\mu}$ if  \[\left \Vert \xi \right \Vert
_{\mathbb{Y}_{\alpha,{\bf{x}},\mu}}^{2}:=\sum_{i=1}^{k}\frac{1}{\mu_{i}^{4}%
}\left \Vert \tilde{\xi}_{i}\rho \right \Vert _{L^{2}\left(  B_{2d_{i}\mu_{i}%
}(0)\right)  }^{2}+\left \Vert \xi \right \Vert _{L^{2}\left(  \Omega \setminus
\Omega^{\prime}\right)  }^{2}<+\infty. \]
Let $\chi_{i}\left(  \left \vert y\right \vert \right)  $ be a smooth
function satisfying $\chi_{i}=1$ in $B_{d_{i}}\left(  0\right)  $, $\chi_{i}=0$ in $\mathbb{R}^2\setminus B_{2d_{i}}\left(  0\right)  $, and $0\leq \chi_{i}\leq1$.  We use the following notations
\begin{align*}
Y_{{\bf{x}},\mu,0}  &  :=-\frac{1}{\mu_{1}}+\sum_{i=1}^{k}\sqrt{\frac
{\rho_{1}}{\rho_{i}}}\frac{2\chi_{i}\left(  y-x_{i}\right)  }{\mu_{i}\left(
1+\mu_{i}^{2}\left \vert y-x_{i}\right \vert ^{2}\right)  },\\
Y_{x_{i},\mu_{i},j}  &  :=\chi_{i}\left(  y-x_{i}\right)  \frac{\mu_{i}%
^{2}\left(  y_{j}-x_{ij}\right)  }{1+\mu_{i}^{2}\left \vert y-x_{i}\right \vert
^{2}},\text{ }i=1,...,k,\text{ }j=1,2,
\end{align*}where $x_i=(x_{i1},x_{i2})$ and $y=(y_1,y_2)$.
The estimations for $Y_{{\bf{x}},\mu,0}$, $Y_{x_{i},\mu_{i},j}$ has been known:
\begin{lemma}\label{kernelapp}\cite{LY1}\[
L_{{\bf{x}},\mu}Y_{{\bf{x}},\mu,0}=O\left(  \mu^{-3}\right)  \text{,
}L_{{\bf{x}},\mu}Y_{x_{i},\mu_{i},j}=O\left(  1\right)  \text{,
}i=1,...,k,\text{ }j=1,2.\]
\end{lemma}
\begin{proof}  See the estimation  (3.8) in \cite{LY1}. \end{proof}
From Lemma \ref{kernelapp}, we see that   $Y_{{\bf{x}},\mu,0}$, $Y_{x_{i},\mu_{i},j}$ are the approximate
kernels for $L_{{\bf{x}},\mu}$.\\
Let
\[
Z_{{\bf{x}},\mu,0}=-\Delta Y_{{\bf{x}},\mu,0}+h_{\mu}\left(  y\right)
Y_{{\bf{x}},\mu,0},
\]
and%
\[
Z_{x_{i},\mu_{i},j}=-\Delta Y_{x_{i},\mu_{i},j}+h_{\mu}\left(  y\right)
Y_{x_{i},\mu_{i},j},\text{ }i=1,...,k\text{, }j=1,2\text{.}%
\]
We define two subspace of $\mathbb{X}_{\alpha,{\bf{x}},\mu}$,
$\mathbb{Y}_{\alpha,{\bf{x}},\mu}$ as
\begin{align*}
E_{{\bf{x}},\mu}  &  :=\{\xi \in \mathbb{X}_{\alpha,{\bf{x}},\mu}\ | \int_{\Omega}Z_{{\bf{x}},\mu,0}\xi dx=\int_{\Omega}Z_{x_{i},\mu_{i},j}%
\xi dx=0\text{, }i=1,...,k,\text{ }j=1,2\},\\
F_{ {\bf{x}},\mu}  &  :=\{\xi \in \mathbb{Y}_{\alpha,{\bf{x}},\mu}\ | \int_{\Omega}Y_{{\bf{x}},\mu,0}\xi dx=\int_{\Omega}Y_{x_{i},\mu_{i},j}%
\xi dx=0\text{, }i=1,...,k,\text{ }j=1,2\}.
\end{align*}
and projection operator $Q_{\bf{x},\mu}:\mathbb{Y}_{\alpha,{\bf{x}},\mu
}\rightarrow F_{\bf{x},\mu}$ by%
\[
Q_{\bf{x},\mu}\xi=\xi-c_{0}Z_{{\bf{x}},\mu,0}-\sum_{j=1}^{2}\sum
_{i=1}^{k}c_{ij}Z_{x_{i},\mu_{i},j},
\]where $c_0, \ c_{i,j}$ are chosen so that $Q_{\bf{x},\mu}\xi\in F_{\bf{x},\mu}$.
For the  projection operator $Q_{\bf{x},\mu}$, we have the following result.
\begin{lemma}\cite[Lemma 3.1]{LY1}\label{pp1}
  There is a constant $C>0$, independent of $\bf{x}$ and $\mu$, such that
  \[\|Q_{{\bf{x}},\mu}u\|_{\mathbb{Y}_{\alpha,{\bf{x}},\mu}}\le C\|u\|_{\mathbb{Y}_{\alpha,{\bf{x}},\mu}}.\]
\end{lemma}
The following lemma will be useful for our arguments.
\begin{lemma}\label{tildeU}
  $\frac{1}{{\e}^2}e^{\tilde{W}_{\bf{x},\mu}}=O(\sum_{i=1}^ke^{u_{x_i,\mu_i}}1_{B_{d_i}}(x_i)+O({\e})(1-\sum_{i=1}^k1_{B_{d_i}(x_i)})).$
\end{lemma}
\begin{proof}  On $B_{d_i}(x_i)$, we see that
\begin{equation*}\begin{aligned}\frac{1}{{\e}^2}e^{\tilde{W}_{\bf{x},\mu}}&=\frac{1}{{\e}^2}e^{w_{\bf{x},\mu}^{\ast}-\int_{\Omega}w_{\bf{x},\mu}^{\ast}(z)dz
+c\left(  w_{\bf{x},\mu}\right)}\\&=\frac{1}{{\e}^2}e^{u_{x_{i},\mu_{i}}\left(  y\right) +\sum_{j\neq i}u_{0,\mu_{j}}\left(  d_j\right)+\Gamma_i-\int_{\Omega}w_{\bf{x},\mu}^{\ast}(z)dz
+c\left(  w_{\bf{x},\mu}\right)  },\end{aligned}\end{equation*} where $\Gamma_i:=8\pi\Big[\gamma \left(  y,x_{i}\right) \left(
1-\frac{1}{{d} \mu_{i}^{2}}\right) +\sum_{j\neq i}( {G}_{{\bf{x}},\mu}\left(  y,x_{j}\right)+\frac{1}{2\pi}\ln d_j) \left(
1-\frac{1}{{d} \mu_{j}^{2}}\right) \Big].$
In the proof of \cite[Proposition 2.1]{LY1}, the following estimations were obtained (see   the estimations (2.13) and  (2.22) in \cite{LY1}):
\[-\int_\Omega w^*_{x_i,\mu_i}(y)dy=2\ln\mu_i+O(1), \quad c(w_{\bf{x},\mu})=-6\ln\mu+O(1).\]
Moreover, $u_{0,\mu_{j}}(d_j)=O(\ln\frac{1}{\mu_i^2})$ on  $\Omega\setminus[B_{d_j}(x_j)]$ and $\mu_i=O(\frac{1}{\sqrt{{\e}}})$ for all $i=1,...,k$.
Thus we get that
\begin{equation*}\begin{aligned}\frac{1}{{\e}^2}e^{\tilde{W}_{\bf{x},\mu}}&= \frac{1}{{\e}^2}e^{u_{x_{i},\mu_{i}}\left(  y\right)}O(\mu^{-2(k-1)+2k-6})=O(e^{u_{x_{i},\mu_{i}}\left(  y\right)}) \ \ \textrm{on}\ \ B_{d_i}(x_i).\end{aligned}\end{equation*}
Similarly we get that \begin{equation*}\begin{aligned}\frac{1}{{\e}^2}e^{\tilde{W}_{\bf{x},\mu}}&=O({\e})\ \ \textrm{on}\ \ \Omega\setminus[\cup_{i}B_{d_i}(x_i)].\end{aligned}\end{equation*}
 \end{proof}

The following invertibility result  for the operator  $Q_{\bf{x},\mu}  L_{{\bf{x}},\mu} $ obtained in \cite{LY1}  is essential for our arguments:
\begin{theorem}\cite[Theorem A.2]{LY1} \label{pp2}
The operator $Q_{\bf{x},\mu}  L_{{\bf{x}},\mu} $ is an isomorphism from $E_{\bf{x},\mu}$ to $F_{\bf{x},\mu}$. Moreover,
if $w\in E_{\bf{x},\mu}$ and $h\in F_{\bf{x},\mu}$ satisfy
\[Q_{\bf{x},\mu}  L_{{\bf{x}},\mu}w=h,\] then there is a constant $C>0$, independent of $\bf{x}$ and $\mu$, such that
\[\|w\|_{\LI}+\|w\|_{\mathbb{X}_{\alpha,{\bf{x}},\mu}}\le C\ln\mu\|h\|_{\mathbb{Y}_{\alpha,{\bf{x}},\mu}}.\]
 \end{theorem}
We define $\tilde{g}_{{\bf{x}},\mu}\left(
\eta\right)  $ as
\[
\tilde{g}_{{\bf{x}},\mu}\left(  \eta\right)  :=h_{\mu}\left(  y\right)
\eta-\frac{1}{{\varepsilon}^{2}}e^{u_0+\tilde{W}_{\bf{x},\mu}+\eta%
}\left(  1-e^{u_0+\tilde{W}_{\bf{x},\mu}+\eta}\right)  -\Delta \tilde
{W}_{\bf{x},\mu}+8k\pi.
\] The function  $\tilde{g}_{{\bf{x}},\mu}\left(
\eta\right)  $ was introduced in \cite{LY1}, and the following estimations were obtained:
\begin{lemma}\cite[Proposition 3.2]{LY1}\label{estg} There is an $\varepsilon_{0}>0$, such that for each $\varepsilon
\in(0,\varepsilon_{0}]$, $\bf{x}$ which is closed to $\bf{q}$ with $\left \vert DG^{\ast
}\left(  {\bf{x}}\right)  \right \vert \leq \frac{C}{\mu}$, and $\mu \in \lbrack
\frac{\beta_{0}}{\sqrt{\varepsilon}},\frac{\beta_{1}}{\sqrt{\varepsilon}}]$, if $\eta, \eta'\in E_{{\bf{x}},\mu}$ satisfies $\|\bar{\eta}\|_{L^\infty(\Omega)}+\|\bar{\eta}\|_{\mathbb{X}_{\alpha,{\bf{x}},\mu}}\le\frac{1}{\mu}$ where $\bar{\eta}=\eta, \eta'$, then we have
 \begin{equation}\label{p12}
\begin{aligned}\left \Vert \tilde{g}%
_{{\bf{x}},\mu}\left(  \eta\right)  \right \Vert _{\mathbb{Y}_{\alpha,{\bf{x}},\mu}}\leq \frac
{C}{\mu^{2-\frac{\alpha}{2}}},\end{aligned}\end{equation}
and
\begin{equation}
\left \Vert \tilde{g}_{{\bf{x}},\mu}\left(  \eta\right)  -\tilde{g}_{2,x,\mu
}\left(  \eta'\right)  \right \Vert _{\mathbb{Y}_{\alpha,{\bf{x}},\mu}}\leq \frac
{C}{\mu }\left \Vert \eta-\eta'\right \Vert _{L^{\infty
}\left(  \Omega \right)  }, \label{n1}%
\end{equation}where $C>0$ is a constant, independent of ${\bf{x}}, \mu, \eta, \eta'$. \end{lemma}
\begin{proof} See  \cite[(3.16)]{LY1} for the estimation \eqref{p12} and \cite[(3.21),(3,22)]{LY1} for the estimation \eqref{n1}.
\end{proof}
We remind that  ${\bf{q}}$ is a non-degenerate critical point of $G^{\ast}\left(  \bf{q} \right)  $ with  $D\left(  \bf{q}\right)  <0$. Now we have the following proposition.
\begin{proposition}\label{pp}
 There is an $\varepsilon_{0}>0$, such that for each $\varepsilon
\in(0,\varepsilon_{0}]$, $\bf{x}$ which is closed to $\bf{q}$ with $\left \vert DG^{\ast
}\left(  {\bf{x}}\right)  \right \vert \leq \frac{C}{\mu}$, and $\mu \in \lbrack
\frac{\beta_{0}}{\sqrt{\varepsilon}},\frac{\beta_{1}}{\sqrt{\varepsilon}}]$, there
exists $\eta_{{\bf{x}},\mu}\in
 E_{\bf{x},\mu}$  satisfying%
\begin{equation}\label{fix2}
Q_{\bf{x},\mu}(  L_{{\bf{x}},\mu}\left(  \eta_{{\bf{x}},\mu}\right))
=Q(g_{{\bf{x}},\mu}(\eta_{{\bf{x}},\mu}))
\end{equation}
Moreover,
\begin{align*}
\left \Vert \eta_{{\bf{x}},\mu}\right \Vert _{L^{\infty}}+\left \Vert \eta_{{\bf{x}},\mu
}\right \Vert _{\mathbb{X}_{\alpha,{\bf{x}},\mu}}   \leq \frac{C\ln \mu}{\mu^{2-\frac{\alpha}{2}}},\end{align*}
where $C>0$ is independent of $\e>0$. Here $\alpha\in(0,\frac{1}{2})$ is the same constant as in
$\mathbb{X}_{\alpha,{\bf{x}},\mu}$ and $\mathbb{Y}_{\alpha,{\bf{x}},\mu}$.
\end{proposition}

\begin{proof}
Define
\[%
\begin{array}
[c]{c}%
S_{\bf{x},\mu}:=\left \{  \eta   \in  E_{\bf{x},\mu}\ |\  \left \Vert \eta\right \Vert _{L^{\infty}\left(  \Omega \right)
}+\left \Vert \eta\right \Vert _{\mathbb{X}_{\alpha,{\bf{x}},\mu}}\leq \frac{1}{\mu}\right \}.
\end{array}
\]
 We denote $\left \Vert \eta\right \Vert _{S_{{\bf{x}},\mu}}:= \left \Vert \eta\right \Vert _{L^{\infty}\left(
\Omega \right)  }+\left \Vert \eta\right \Vert _{\mathbb{X}_{\alpha,{\bf{x}},\mu}}$ as the
norm in $S_{\bf{x},\mu}$. We consider the following   mapping
\[%
\begin{array}
[c]{c}%
B_{\bf{x},\mu}:\eta  \rightarrow
(Q_{\bf{x},\mu}L_{{\bf{x}},\mu})^{-1}[Q_{\bf{x},\mu}{g}_{{\bf{x}},\mu}\left( \eta_{{\bf{x}},\mu}\right)] .
\end{array}
\]

Step 1. First, we claim that  $B_{\bf{x},\mu}$  maps $S_{\bf{x},\mu}$ to  $S_{\bf{x},\mu}$.

In view of   Theorem \ref{pp2},  and  Lemma \ref{pp1}, we have for some constant $C>0$, independent of $\e>0$,
\begin{align*}
\left \Vert B_{\bf{x},\mu}\left(  \eta\right)  \right \Vert
_{S_{\bf{x},\mu}}  \leq C\ln
\mu \left \Vert {g}_{{\bf{x}},\mu}\left(\eta\right)  \right \Vert
_{\mathbb{Y}_{\alpha,{\bf{x}},\mu}}.
\end{align*}

 From the definition of $\tilde{g}_{{\bf{x}},\mu}\left(
\eta\right)  $ and $G(1+s-e^{s})=s$, we see that
\begin{equation*}
\begin{aligned}
&   {g}_{{\bf{x}},\mu}\left(\eta\right)   -\tilde{g}_{{\bf{x}},\mu}\left(\eta\right)
\\&=\frac{1}{{\varepsilon}^{2}}e^{u_0+\tilde{W}_{\bf{x},\mu}+\eta}
 (1-e^{u_0+\tilde{W}_{\bf{x},\mu}+\eta})
  -\frac{1}{{\varepsilon}^{2}}e^{G(1+u_0+\tilde{W}_{\bf{x},\mu}+\eta)}
 (1-e^{G(1+u_0+\tilde{W}_{\bf{x},\mu}+\eta)})^2
 \\&=\frac{1}{{\varepsilon}^{2}}e^{u_0+\tilde{W}_{\bf{x},\mu}+\eta}
 (1-e^{u_0+\tilde{W}_{\bf{x},\mu}+\eta})
   -\frac{1}{{\varepsilon}^{2}}e^{ u_0+\tilde{W}_{\bf{x},\mu}+\eta }
 (1-e^{ u_0+\tilde{W}_{\bf{x},\mu}+\eta})^2
 \\& +\frac{1}{{\varepsilon}^{2}}e^{G(1+u_0+\tilde{W}_{\bf{x},\mu}+\eta-e^{u_0+\tilde{W}_{\bf{x},\mu}+\eta})}
 (1-e^{G(1+u_0+\tilde{W}_{\bf{x},\mu}+\eta-e^{u_0+\tilde{W}_{\bf{x},\mu}+\eta})})^2
 \\&-\frac{1}{{\varepsilon}^{2}}e^{G(1+u_0+\tilde{W}_{\bf{x},\mu}+\eta)}
 (1-e^{G(1+u_0+\tilde{W}_{\bf{x},\mu}+\eta)})^2
 .
\end{aligned}\end{equation*}If $t=1+s-e^s$, then we see that $G(t)=G(1+s-e^{s})=s$, $dt=(1-e^{s})ds$, and $\frac{d}{dt}G(t)=\frac{ds}{dt}=\frac{1}{1-e^{s}}=\frac{1}{1-e^{G(t)}}$.\\
Since $\frac{d}{dt}e^{G(t)}(1-e^{G(t)})^2=e^{G(t)}(1-e^{G(t)})(1-3e^{G(t)})\frac{d}{dt}G(t)=e^{G(t)}(1-3e^{G(t)})$, we have for some $\theta\in(0,1)$,
\begin{equation*}
\begin{aligned}
&   {g}_{{\bf{x}},\mu}\left(\eta\right)   -\tilde{g}_{{\bf{x}},\mu}\left(\eta\right)
=\frac{1}{{\varepsilon}^{2}}e^{2u_0+2\tilde{W}_{\bf{x},\mu}+2\eta}
 (1-e^{u_0+\tilde{W}_{\bf{x},\mu}+\eta})
 \\&-\frac{1}{{\varepsilon}^{2}}e^{G(1+u_0+\tilde{W}_{\bf{x},\mu}+\eta-\theta e^{u_0+\tilde{W}_{\bf{x},\mu}+\eta})}
 (1-3e^{G(1+u_0+\tilde{W}_{\bf{x},\mu}+\eta-\theta e^{u_0+\tilde{W}_{\bf{x},\mu}+\eta})})e^{u_0+\tilde{W}_{\bf{x},\mu}+\eta}.\end{aligned}\end{equation*}
Since $\frac{d}{dt}e^{G(t)}(1-3e^{G(t)})=\frac{e^{G(t)}(1-6e^{G(t)})}{1-e^{G(t)}}$ and $\frac{d}{dt}G(t)= \frac{1}{1-e^{G(t)}}$, we have for some $\theta'\in(\theta,1)$ and $\theta''\in(\theta',1)$,
 \begin{equation*}
\begin{aligned}
&   {g}_{{\bf{x}},\mu}\left(\eta\right)   -\tilde{g}_{{\bf{x}},\mu}\left(\eta\right)
\\&=\frac{1}{{\varepsilon}^{2}}e^{2u_0+2\tilde{W}_{\bf{x},\mu}+2\eta}
 (1-e^{u_0+\tilde{W}_{\bf{x},\mu}+\eta})
 \\&-\frac{1}{{\varepsilon}^{2}}e^{G(1+u_0+\tilde{W}_{\bf{x},\mu}+\eta-  e^{u_0+\tilde{W}_{\bf{x},\mu}+\eta})}
 (1-3e^{G(1+u_0+\tilde{W}_{\bf{x},\mu}+\eta-  e^{u_0+\tilde{W}_{\bf{x},\mu}+\eta})})e^{u_0+\tilde{W}_{\bf{x},\mu}+\eta}
\\& +\frac{1}{{\varepsilon}^{2}}e^{G(1+u_0+\tilde{W}_{\bf{x},\mu}+\eta-  e^{u_0+\tilde{W}_{\bf{x},\mu}+\eta})}
 (1-3e^{G(1+u_0+\tilde{W}_{\bf{x},\mu}+\eta-  e^{u_0+\tilde{W}_{\bf{x},\mu}+\eta})})e^{u_0+\tilde{W}_{\bf{x},\mu}+\eta}
\\& -\frac{1}{{\varepsilon}^{2}}e^{G(1+u_0+\tilde{W}_{\bf{x},\mu}+\eta-\theta e^{u_0+\tilde{W}_{\bf{x},\mu}+\eta})}
 (1-3e^{G(1+u_0+\tilde{W}_{\bf{x},\mu}+\eta-\theta e^{u_0+\tilde{W}_{\bf{x},\mu}+\eta})})e^{u_0+\tilde{W}_{\bf{x},\mu}+\eta}
 \\&=\frac{1}{{\varepsilon}^{2}}e^{2u_0+2\tilde{W}_{\bf{x},\mu}+2\eta}
 (1-e^{u_0+\tilde{W}_{\bf{x},\mu}+\eta})
 -\frac{1}{{\varepsilon}^{2}}e^{ 2u_0+2\tilde{W}_{\bf{x},\mu}+2\eta}
 (1-3e^{u_0+\tilde{W}_{\bf{x},\mu}+\eta})
\\& +\frac{e^{G(1+u_0+\tilde{W}_{\bf{x},\mu}+\eta-  \theta'e^{u_0+\tilde{W}_{\bf{x},\mu}+\eta})}
 (1-6e^{G(1+u_0+\tilde{W}_{\bf{x},\mu}+\eta- \theta' e^{u_0+\tilde{W}_{\bf{x},\mu}+\eta})})(\theta-1) e^{2u_0+2\tilde{W}_{\bf{x},\mu}+2\eta}}{\e^2(1-e^{G(1+u_0+\tilde{W}_{\bf{x},\mu}+\eta- \theta' e^{u_0+\tilde{W}_{\bf{x},\mu}+\eta})})},
\end{aligned}\end{equation*} and thus
\begin{equation*}
\begin{aligned}
&   {g}_{{\bf{x}},\mu}\left(\eta\right)   -\tilde{g}_{{\bf{x}},\mu}\left(\eta\right)
=\frac{2e^{3u_0+3\tilde{W}_{\bf{x},\mu}+3\eta}}{{\varepsilon}^{2}}
\\& +\frac{\{e^{G(1+u_0+\tilde{W}_{\bf{x},\mu}+\eta-  \theta'e^{u_0+\tilde{W}_{\bf{x},\mu}+\eta})}-e^{G(1+u_0+\tilde{W}_{\bf{x},\mu}+\eta-  e^{u_0+\tilde{W}_{\bf{x},\mu}+\eta})}\}
 }{\e^2(1-e^{G(1+u_0+\tilde{W}_{\bf{x},\mu}+\eta- \theta' e^{u_0+\tilde{W}_{\bf{x},\mu}+\eta})})}\\&\times (1-6e^{G(1+u_0+\tilde{W}_{\bf{x},\mu}+\eta- \theta' e^{u_0+\tilde{W}_{\bf{x},\mu}+\eta})})(\theta-1) e^{2u_0+2\tilde{W}_{\bf{x},\mu}+2\eta}
  \\& +\frac{ (1-6e^{G(1+u_0+\tilde{W}_{\bf{x},\mu}+\eta- \theta' e^{u_0+\tilde{W}_{\bf{x},\mu}+\eta})})(\theta-1) e^{3u_0+3\tilde{W}_{\bf{x},\mu}+3\eta}}{\e^2(1-e^{G(1+u_0+\tilde{W}_{\bf{x},\mu}+\eta- \theta' e^{u_0+\tilde{W}_{\bf{x},\mu}+\eta})})}
  \\&=\frac{2e^{3u_0+3\tilde{W}_{\bf{x},\mu}+3\eta}}{{\varepsilon}^{2}}
\\& +\frac{(1-6e^{G(1+u_0+\tilde{W}_{\bf{x},\mu}+\eta- \theta' e^{u_0+\tilde{W}_{\bf{x},\mu}+\eta})})(1-\theta')(\theta-1) e^{3u_0+3\tilde{W}_{\bf{x},\mu}+3\eta} }{\e^2(1-e^{G(1+u_0+\tilde{W}_{\bf{x},\mu}+\eta- \theta' e^{u_0+\tilde{W}_{\bf{x},\mu}+\eta})})(1-e^{G(1+u_0+\tilde{W}_{\bf{x},\mu}+\eta- \theta'' e^{u_0+\tilde{W}_{\bf{x},\mu}+\eta})})}
  \\& +\frac{ (1-6e^{G(1+u_0+\tilde{W}_{\bf{x},\mu}+\eta- \theta' e^{u_0+\tilde{W}_{\bf{x},\mu}+\eta})})(\theta-1) e^{3u_0+3\tilde{W}_{\bf{x},\mu}+3\eta}}{\e^2(1-e^{G(1+u_0+\tilde{W}_{\bf{x},\mu}+\eta- \theta' e^{u_0+\tilde{W}_{\bf{x},\mu}+\eta})})}.
  \end{aligned}\end{equation*}
By  Lemma \ref{tildeU} and $G(-\infty)=-\infty$, we see that as $\e\to0$,
$e^{\tilde{W}_{\bf{x},\mu}}=O(\e)$ and $u_0+\tilde{W}_{\bf{x},\mu}+\eta\to\ -\infty$ uniformly on $\Omega$, which implies that
\begin{equation}
\begin{aligned}\label{totaldiff}
&   {g}_{{\bf{x}},\mu}\left(\eta\right)   -\tilde{g}_{{\bf{x}},\mu}\left(\eta\right)
= O\Big( \frac{e^{3u_0+3\tilde{W}_{\bf{x},\mu}+3\eta}}{\e^2}\Big)\ \ \textrm{on}\ \ \Omega.\end{aligned}\end{equation}
  In view of  Lemma \ref{tildeU}, we see that
$e^{\tilde{W}_{\bf{x},\mu}}=O(\e^3)$ on $\Omega\setminus[(\cup_{i=1}^k B_{d_i}(x_i)]$,  and thus
\begin{equation}
\begin{aligned}\label{l1}
&   {g}_{{\bf{x}},\mu}\left(\eta\right)   -\tilde{g}_{{\bf{x}},\mu}\left(\eta\right)=O(\e^7) \ \textrm{on} \ \Omega\setminus[(\cup_{i=1}^k B_{d_i}(x_i)].
\end{aligned}\end{equation}
On $B_{d_i}(x_i)$, from Lemma \ref{tildeU}, we see that $e^{\tilde{W}_{\bf{x},\mu}}=O\Big(\frac{\mu_i^2\e^2}{(1+\mu_i^2|y-x_i|^2)^2}\Big)=O\Big(\frac{\e}{(1+\mu_i^2|y-x_i|^2)^2}\Big)$ and
\begin{equation*}
\begin{aligned}
&   {g}_{{\bf{x}},\mu}\left(\eta\right)   -\tilde{g}_{{\bf{x}},\mu}\left(\eta\right)
  =O\Big(\frac{\e}{(1+\mu_i^2|y-x_i|^2)^6}\Big),
\end{aligned}\end{equation*}
which  implies that
\begin{equation}\label{l3}
\begin{aligned}
& \frac{1}{\mu_i^2} \|[ {g}_{{\bf{x}},\mu}\left(\eta\right)   -\tilde{g}_{{\bf{x}},\mu}\left(\eta\right)](x_i+\mu_i^{-1}y)
(1+|y|)^{1+\frac{\alpha}{2}}\|_{L^2(B_{d_i\mu_i}(0))}
\\&  = \Big\|\frac{O(\e^2)(1+|y|)^{1+\frac{\alpha}{2}}}{(1+|y|^2)^6}
\Big\|_{L^2(B_{d_i\mu_i}(0))}
 = O(\e^2).
\end{aligned}\end{equation}
From the above arguments \eqref{l1}-\eqref{l3}, we have \begin{equation}\label{p11}
\begin{aligned}\left \Vert {g}_{{\bf{x}},\mu}\left(\eta\right)
-\tilde{g}_{{\bf{x}},\mu}\left(  \eta\right)  \right \Vert _{\mathbb{Y}_{\alpha,{\bf{x}},\mu}%
}\leq \frac{C}{\mu^{2}}.\end{aligned}\end{equation}   Combining   Lemma \ref{estg} and  the estimation \eqref{p11} together, we obtain%
\begin{equation}
\left \Vert {g}_{{\bf{x}},\mu}\left(\eta\right)  \right \Vert
_{\mathbb{Y}_{\alpha,{\bf{x}},\mu}}\leq \frac{C}{\mu^{2-\frac{\alpha}{2}}}. \label{p09}%
\end{equation}
By \eqref{p09}, we see that for large $\mu>0$ (i.e. for small $\e>0$), $B_{\bf{x},\mu}$ maps $S_{\bf{x},\mu}$ to $S_{\bf{x},\mu}$.

Step 2. Now we claim that $B_{\bf{x},\mu}$  is a contraction map.

In view of   Theorem \ref{pp2},  and  Lemma \ref{pp1}, there is a constant $C>0$, independent of $\e>0$, satisfying  for any
$ \eta   ,\eta'   \in S_{\bf{x},\mu}$,
\begin{equation}\begin{aligned}\label{p06}
&  \left \Vert B_{\bf{x},\mu}\left(  \eta\right)  -B_{\bf{x},\mu}\left(
\eta' \right)  \right \Vert _{S_{\bf{x},\mu}%
}  \leq C \ln \mu \left \Vert g_{\bf{x},\mu}\left(  \eta\right)  -g_{\bf{x},\mu}( \eta' )
\right \Vert _{\mathbb{Y}_{\alpha,{\bf{x}},\mu}}.
\end{aligned}\end{equation}
To estimate $\left \Vert {g}_{{\bf{x}},\mu}\left(\eta\right)
-  g_{{\bf{x}},\mu}(\eta')    \right \Vert
_{\mathbb{Y}_{\alpha,{\bf{x}},\mu}}$, we observe that%
\begin{align*}
\left \Vert {g}_{{\bf{x}},\mu}\left(\eta\right) -  g_{{\bf{x}},\mu}(\eta')    \right \Vert _{\mathbb{Y}_{\alpha,{\bf{x}},\mu}}
& \leq \left \Vert {g}_{{\bf{x}},\mu}\left(\eta\right)  -\tilde
{g}_{{\bf{x}},\mu}\left(  \eta\right)  -\left(   g_{{\bf{x}},\mu}(\eta')    -\tilde{g}_{{\bf{x}},\mu}\left(  \eta
^{\prime}\right)  \right)  \right \Vert _{\mathbb{Y}_{\alpha,{\bf{x}},\mu}}\\
&  +\left \Vert \tilde{g}_{{\bf{x}},\mu}\left(  \eta\right)  -\tilde{g}%
_{{\bf{x}},\mu}\left(  \eta'\right)  \right \Vert _{\mathbb{Y}_{\alpha,{\bf{x}},\mu}}.
\end{align*}
We see that
\begin{align*}
&  {g}_{{\bf{x}},\mu}\left(\eta\right)  -\tilde
{g}_{{\bf{x}},\mu}\left(  \eta\right)  -\left(   g_{{\bf{x}},\mu}(\eta')    -\tilde{g}_{{\bf{x}},\mu}\left(  \eta
^{\prime}\right)  \right)
  \\&=  \frac{1}{\e^2}e^{2u_0+2\tilde{W}_{\bf{x},\mu}+2\eta}(1-e^{u_0+\tilde{W}_{\bf{x},\mu}+\eta}) -\frac{1}{\e^2}e^{2u_0+2\tilde{W}_{\bf{x},\mu}+2\eta'}(1-e^{u_0+\tilde{W}_{\bf{x},\mu}+\eta'})
  \\& +\frac{1}{\e^2}e^{G(1+u_0+\tilde{W}_{\bf{x},\mu}+\eta-e^{u_0+\tilde{W}_{\bf{x},\mu}+\eta})}(1-e^{G(1+u_0+\tilde{W}_{\bf{x},\mu}+\eta-e^{u_0+\tilde{W}_{\bf{x},\mu}+\eta} )})^2
  \\&-\frac{1}{\e^2}e^{G(1+u_0+\tilde{W}_{\bf{x},\mu}+\eta'-e^{u_0+\tilde{W}_{\bf{x},\mu}+\eta'})}(1-e^{G(1+u_0+\tilde{W}_{\bf{x},\mu}+\eta'-e^{u_0+\tilde{W}_{\bf{x},\mu}+\eta'} )})^2
  \\&-\frac{1}{\e^2}e^{G(1+u_0+\tilde{W}_{\bf{x},\mu}+\eta)}(1- e^{G(1+u_0+\tilde{W}_{\bf{x},\mu}+\eta )})^2
  +\frac{1}{\e^2}e^{G(1+u_0+\tilde{W}_{\bf{x},\mu}+\eta')}(1- e^{G(1+u_0+\tilde{W}_{\bf{x},\mu}+\eta' )})^2.\end{align*}
  Then for some numbers $\xi_0, \xi_1, \xi_2, \xi_3$ between $\eta$ and $\eta'$, and some $\theta\in(0,1)$, $\theta'\in(\theta,1)$, we have
  \begin{align*}
&  {g}_{{\bf{x}},\mu}\left(\eta\right)  -\tilde
{g}_{{\bf{x}},\mu}\left(  \eta\right)  -\left(   g_{{\bf{x}},\mu}(\eta')    -\tilde{g}_{{\bf{x}},\mu}\left(  \eta
^{\prime}\right)  \right)
  \\&=  \frac{1}{\e^2}e^{2u_0+2\tilde{W}_{\bf{x},\mu}+2\xi_0}(2-3e^{u_0+\tilde{W}_{\bf{x},\mu}+\xi_0})(\eta-\eta')
  \\& +\frac{ e^{G(1+u_0+\tilde{W}_{\bf{x},\mu}+\xi_1-e^{u_0+\tilde{W}_{\bf{x},\mu}+\xi_1})}}{\e^2}(1-3e^{G(1+u_0+\tilde{W}_{\bf{x},\mu}+\xi_1-e^{u_0+\tilde{W}_{\bf{x},\mu}+\xi_1} )})(1-e^{u_0+\tilde{W}_{\bf{x},\mu}+\xi_1})(\eta-\eta')
  \\&+\frac{1}{\e^2}e^{G(1+u_0+\tilde{W}_{\bf{x},\mu}+\xi_2)}(1-3e^{G(1+u_0+\tilde{W}_{\bf{x},\mu}+\xi_2)})(\eta'-\eta)
  \\&= \frac{1}{\e^2}e^{2u_0+2\tilde{W}_{\bf{x},\mu}+2\xi_0}(2-3e^{u_0+\tilde{W}_{\bf{x},\mu}+\xi_0})(\eta-\eta')
  \\& +\Big\{\frac{ e^{G(1+u_0+\tilde{W}_{\bf{x},\mu}+\xi_1-e^{u_0+\tilde{W}_{\bf{x},\mu}+\xi_1})}}{\e^2}(1-3e^{G(1+u_0+\tilde{W}_{\bf{x},\mu}+\xi_1-e^{u_0+\tilde{W}_{\bf{x},\mu}+\xi_1} )})\\&-\frac{ e^{G(1+u_0+\tilde{W}_{\bf{x},\mu}+\xi_2)}}{\e^2}(1-3e^{G(1+u_0+\tilde{W}_{\bf{x},\mu}+\xi_2 )})\Big\}(\eta-\eta')
  \\&-\frac{1}{\e^2}e^{2u_0+2\tilde{W}_{\bf{x},\mu}+2\xi_1}(1-3e^{G(u_0+\tilde{W}_{\bf{x},\mu}+\xi_1-e^{u_0+\tilde{W}_{\bf{x},\mu}+\xi_1})})(\eta-\eta')
  \\&=O\Big(\frac{e^{2u_0+2\tilde{W}_{\bf{x},\mu}}}{\e^2}\Big)(\eta-\eta')
  \\&+  \frac{(1-6e^{G(1+u_0+\tilde{W}_{\bf{x},\mu}+\xi_3-\theta e^{u_0+\tilde{W}_{\bf{x},\mu}+\xi_1})})}{\e^2(1-e^{G(1+u_0+\tilde{W}_{\bf{x},\mu}+\xi_3-\theta e^{u_0+\tilde{W}_{\bf{x},\mu}+\xi_1})})}(\eta-\eta')(\xi_1-\xi_2-e^{u_0+\tilde{W}_{\bf{x},\mu}+\xi_1})
  \\&\times \Big\{\frac{e^{u_0+\tilde{W}_{\bf{x},\mu}+\xi_3}-\theta e^{u_0+\tilde{W}_{\bf{x},\mu}+\xi_1}}{(1-e^{G(1+u_0+\tilde{W}_{\bf{x},\mu}+\xi_3-\theta' e^{u_0+\tilde{W}_{\bf{x},\mu}+\xi_1})})}+e^{u_0+\tilde{W}_{\bf{x},\mu}+\xi_3}\Big\}
   \\&=O\Big(\frac{e^{2u_0+2\tilde{W}_{\bf{x},\mu}}}{\e^2}\Big)(\eta-\eta')+ O\Big(\frac{e^{u_0+\tilde{W}_{\bf{x},\mu}}}{\e^2}\Big)(\eta-\eta')^2.\end{align*}
On $\Omega\setminus [\cup B_{d_i}(x_i)]$,   we see that from Lemma \ref{tildeU},
\begin{equation}
\begin{aligned}\label{n2}
& \| {g}_{{\bf{x}},\mu}\left(\eta\right)  -\tilde
{g}_{{\bf{x}},\mu}\left(  \eta\right)  -\left(   g_{{\bf{x}},\mu}(\eta')    -\tilde{g}_{{\bf{x}},\mu}\left(  \eta
^{\prime}\right)  \right)  \|_{L^\infty(\Omega\setminus [\cup B_{d_i}(x_i)])}
\\&\le  O(\e)(\|\eta-\eta'\|_{L^\infty(\Omega)}+(\|\eta\|_{L^\infty(\Omega)}+\|\eta'\|_{L^\infty(\Omega)})\|\eta-\eta'\|_{L^\infty(\Omega)})
.
\end{aligned}\end{equation}
On $B_{d_i}(x_i)$, from Lemma \ref{tildeU}, we see that $e^{\tilde{W}_{\bf{x},\mu}}=O\Big(\frac{\mu_i^2\e^2}{(1+\mu_i^2|y-x_i|^2)^2}\Big)=O\Big(\frac{\e}{(1+\mu_i^2|y-x_i|^2)^2}\Big)$ and\begin{equation}
\begin{aligned}\label{n3}
& \frac{1}{\mu_i^2} \|[ {g}_{{\bf{x}},\mu}\left(\eta\right)  -\tilde
{g}_{{\bf{x}},\mu}\left(  \eta\right)  -\left(   g_{{\bf{x}},\mu}(\eta')    -\tilde{g}_{{\bf{x}},\mu}\left(  \eta
^{\prime}\right)  \right)](x_i+\mu_i^{-1}y)
 (1+|y|)^{1+\frac{\alpha}{2}}\|_{L^2(B_{d_i\mu_i}(0))}
\\&  =(\|\eta\|_{L^\infty(\Omega)}+\|\eta'\|_{L^\infty(\Omega)}+O(\e))\|\eta-\eta'\|_{L^\infty(\Omega)}.
\end{aligned}\end{equation}
From  Lemma \ref{estg}, and the above estimations \eqref{n2}-\eqref{n3}, we have
\begin{equation}\label{p07}
\begin{aligned}
&  \left \Vert {g}_{{\bf{x}},\mu}\left(\eta\right) -  g_{{\bf{x}},\mu}(\eta')    \right \Vert _{\mathbb{Y}_{\alpha,{\bf{x}},\mu}}
\le(\|\eta\|_{L^\infty(\Omega)}+\|\eta'\|_{L^\infty(\Omega)}+O(\e^{\frac{1}{2}}))\|\eta-\eta'\|_{L^\infty(\Omega)}.
\end{aligned}\end{equation}
In view of the estimations (\ref{p06})-(\ref{p07}), we obtain that  $B_{\bf{x},\mu}$ is a
contraction map on $S_{\bf{x},\mu}$.

Step 3.
In view of Step 1, Step 2, and  contraction mapping
theorem, there exists a unique solution $ \eta_{{\bf{x}},\mu}  \in S_{\bf{x},\mu}$ of (\ref{fix2}).
Moreover,  from Theorem \ref{pp2},  Lemma \ref{pp1}, and \eqref{p09}, we obtain that \begin{align*}
\left \Vert \eta_{{\bf{x}},\mu}\right \Vert _{L^{\infty}}+\left \Vert \eta_{{\bf{x}},\mu
}\right \Vert _{\mathbb{X}_{\alpha,{\bf{x}},\mu}}   \leq \frac{C\ln \mu}{\mu^{2-\frac{\alpha}{2}}},
\end{align*}
where $C>0$ is independent of $\e>0$. Now we complete the proof of Proposition \ref{pp}.
\end{proof}

By Proposition \ref{pp}, we get that for any $\mu\in \left[  \frac{\beta_{0}}{\sqrt{\varepsilon}}%
,\frac{\beta_{1}}{\sqrt{\varepsilon}}\right]  $, and any $\bf{x}$ close to $\bf{q}$, where ${\bf{q}}$ is a non-degenerate critical point of $G^{\ast}\left(  \bf{q} \right)  $ with  $D\left(  \bf{q}\right)  <0$,
 there is $\eta_{{\bf{x}},\mu} \in S_{\bf{x},\mu}$ such that%
\begin{equation}\begin{aligned}
\label{final}
&\Delta(\tilde{W}_{\bf{x},\mu}+\eta_{{\bf{x}},\mu})
\\&=- \frac{1}{{\varepsilon}^{2}}e^{G(u_0+\tilde
{W}_{\bf{x},\mu}+\eta_{{\bf{x}},\mu})}(1-e^{G(u_0+\tilde{W}_{x,\mu
}+\eta_{{\bf{x}},\mu})})^2 +8k\pi+c_{0}Z_{{\bf{x}},\mu,0}+\sum_{i=1}^{k}\sum_{j=1}^{2}
c_{ij}Z_{x,\mu,j},\end{aligned}
\end{equation}
where $c_{0}$, $c_{ij}$ are constants satisfying \[ L_{{\bf{x}},\mu}\left(  \eta_{{\bf{x}},\mu}\right)
-g_{{\bf{x}},\mu}\left(  \eta_{{\bf{x}},\mu}\right)-c_{0}Z_{{\bf{x}},\mu,0}-\sum_{i=1}^{k}\sum_{j=1}^{2}
c_{ij}Z_{x,\mu,j} \in F_{{\bf{x}},\mu}.\]
In the following, we will choose $\bf{x},\mu$ suitably ( depending on $\varepsilon$
) such that the corresponding $c_{0}$, $c_{ij}$ are zero and hence the
solution $ \eta_{{\bf{x}},\mu} $ is exactly the solution to
 (\ref{err}) which implies that $ u_\e = 1+u_0+\tilde{W}_{\bf{x},\mu}+\eta_{{\bf{x}},\mu}$ is a solution to \eqref{1}.
It is standard to prove the following lemma.
\begin{lemma}\label{ppp}
If
\begin{equation*}\begin{aligned}
&\int_{\Omega}\Big[ \Delta \eta_{{\bf{x}},\mu}
+\frac{1}{{\varepsilon}^{2}}e^{G(u_0+\tilde
{W}_{\bf{x},\mu}+\eta_{{\bf{x}},\mu})}(1-e^{G(u_0+\tilde{W}_{x,\mu
}+\eta_{{\bf{x}},\mu})})^2+\Delta\tilde{W}_{\bf{x},\mu}-8k\pi \Big]
Y_{x_{i},\mu_{i},j}dx=0,
\end{aligned}
\end{equation*}
and%
\begin{equation*}\begin{aligned}
&\int_{\Omega}\Big[ \Delta \eta_{{\bf{x}},\mu}
+\frac{1}{{\varepsilon}^{2}}e^{G(u_0+\tilde
{W}_{\bf{x},\mu}+\eta_{{\bf{x}},\mu})}(1-e^{G(u_0+\tilde{W}_{x,\mu
}+\eta_{{\bf{x}},\mu})})^2+\Delta\tilde{W}_{\bf{x},\mu}-8k\pi  \Big]  Y_{{\bf{x}},\mu,0}dx=0,
\end{aligned}
\end{equation*}
then $c_{0}=c_{ij}=0$ for $i=1,...,k$ and $j=1,2$.
\end{lemma}
Let $x_i=(x_{i1},x_{i2})$.
By using the proof of \cite[Theorem 1.2]{LY1}, we get the following result.
\begin{proposition}\label{pppp}
We have%
\begin{equation}\begin{aligned}
&\int_{\Omega}\Big[ \Delta \eta_{{\bf{x}},\mu}
+\frac{1}{{\varepsilon}^{2}}e^{G(u_0+\tilde
{W}_{\bf{x},\mu}+\eta_{{\bf{x}},\mu})}(1-e^{G(u_0+\tilde{W}_{x,\mu
}+\eta_{{\bf{x}},\mu})})^2+\Delta\tilde{W}_{\bf{x},\mu}-8k\pi \Big]
Y_{x_{i},\mu_{i},j}dx
\\&=A_0\frac{\partial G^{\ast}\left(  {\bf{x}}\right)  }{\partial x_{ij}}+O\left(
\frac{\ln \mu}{\mu^{2-\frac{\alpha}{2}}}\right)\ \ \textrm{for}\ \ j=1,2,
\label{1p1}%
\end{aligned}
\end{equation}
and%
\begin{equation}\begin{aligned}
&\int_{\Omega}\Big[\Delta \eta_{{\bf{x}},\mu}
+\frac{1}{{\varepsilon}^{2}}e^{G(u_0+\tilde
{W}_{\bf{x},\mu}+\eta_{{\bf{x}},\mu})}(1-e^{G(u_0+\tilde{W}_{x,\mu
}+\eta_{{\bf{x}},\mu})})^2+\Delta\tilde{W}_{\bf{x},\mu}-8k\pi  \Big] Y_{{\bf{x}},\mu,0}dx
\\&=\frac{8}{\rho_{1}\mu^{3}}\left(
\sum_{i=1}^{k}
\rho_{i}\left(  \int_{\Omega_{i}\setminus B_{d_i}\left(  x_i\right)  }%
\frac{e^{f_{{\bf{x}},i}}-1}{\left \vert y-x_{i}\right \vert ^{4}}-\int_{\mathbb{R}^2\setminus
\Omega_{i}}\frac{1}{\left \vert y-x_{i}\right \vert ^{4}}\right)  \right)
\\&+B_0{\varepsilon}^{2}\mu
+\frac{1}{\mu^{3}}O\left(  \left \vert DG^{\ast}\left(  {\bf{x}}\right)  \right \vert
^{2}\ln \mu+\delta^{2}\right)  +O\left(  \frac{1}{\mu^{5}}\right),
\label{1p2}%
\end{aligned}
\end{equation}
where $A_0, B_0>0$ are  constants, $\delta>0$ is any small constant, $\Omega
_{1},...,\Omega_{k}$ are any open set with $\Omega_{i}\cap \Omega_{j}%
=\emptyset$ for $i\neq j$, $\cup_{i=1}^{k}\bar{\Omega}_{i}=\Omega$, $B_{d_{i}%
}\left(  x_{i}\right)  \subset\subset \Omega_{i}$, $i=1,...,k$, and%
\begin{equation*}\begin{aligned}
&f_{{\bf{x}},i}\left(  y\right)  =  8\pi \left(  \gamma \left(  y,x_{i}\right)
-\gamma \left(  x_{i},x_{i}\right)  +%
%TCIMACRO{\dsum \limits_{j\neq i}^{k}}%
%BeginExpansion
{\displaystyle \sum \limits_{j\neq i}^{k}}
%EndExpansion
\left(  G\left(  y,x_{j}\right)  -G\left(  x_{i},x_{j}\right)  \right)
\right) +u_0\left(  y\right)  -u_0\left(  x_{i}\right).
\end{aligned}\end{equation*}

\end{proposition}

\begin{proof}
In \cite{LY1}, if  $\eta_{{\bf{x}},\mu}\in \mathbb{X}_{\alpha,{\bf{x}},\mu}$ satisfies $\left \Vert \eta_{{\bf{x}},\mu}\right \Vert _{L^{\infty}\left(  \Omega \right)
}+\left \Vert \eta_{{\bf{x}},\mu}\right \Vert _{\mathbb{X}_{\alpha,{\bf{x}},\mu}}\leq \frac{C\ln \mu}{\mu^{2-\frac{\alpha}{2}}}$, then the followings  hold:
\begin{equation}\begin{aligned}
&\int_{\Omega}\left(  \Delta \left(  \tilde{W}_{\bf{x},\mu}+\eta_{{\bf{x}},\mu}\right)
+\frac{1}{{\varepsilon}^{2}}e^{u_0+\tilde{W}_{\bf{x},\mu}+\eta_{{\bf{x}},\mu}%
}\left(  1-e^{u_0+\tilde{W}_{\bf{x},\mu}+\eta_{{\bf{x}},\mu}}\right)  -8k\pi \right)
Y_{x_{i},\mu_{i},j}dx\\&
=A_0\frac{\partial G^{\ast}\left(  {\bf{x}}\right)  }{\partial x_{ij}}+O\left(
\frac{\ln \mu}{\mu^{2-\frac{\alpha}{2}}}\right)\ \ \textrm{for}\ \ j=1,2,
\end{aligned}\label{1p11}%
\end{equation}
and%
\begin{equation}\begin{aligned}&
\int_{\Omega}\left(  \Delta \left(  \tilde{W}_{\bf{x},\mu}+\eta_{{\bf{x}},\mu}\right)  +\frac
{1}{{\varepsilon}^{2}}e^{u_0+\tilde{W}_{\bf{x},\mu}+\eta_{{\bf{x}},\mu}}\left(
1-e^{u_0+\tilde{W}_{\bf{x},\mu}+\eta_{{\bf{x}},\mu}}\right)  -8k\pi \right)  Y_{{\bf{x}},\mu,0}dx\\&
=\frac{8}{\rho_{1}\mu^{3}}\left(
%TCIMACRO{\dsum \limits_{i=1}^{k}}%
%BeginExpansion
{\displaystyle \sum \limits_{i=1}^{k}}
%EndExpansion
\rho_{i}\left(  \int_{\Omega_{i}\setminus B_{d_i}\left(  x_i\right)  }%
\frac{e^{f_{{\bf{x}},i}}-1}{\left \vert y-x_{i}\right \vert ^{4}}-\int_{\mathbb{R}^2\setminus
\Omega_{i}}\frac{1}{\left \vert y-x_{i}\right \vert ^{4}}\right)  \right)
+B_0{\varepsilon}^{2}\mu \\&
+\frac{1}{\mu^{3}}O\left(  \left \vert DG^{\ast}\left(  {\bf{x}}\right)  \right \vert
^{2}\ln \mu+\delta^{2}\right)  +O\left(  \frac{1}{\mu^{5}}\right)  .
\end{aligned}\label{1p12}%
\end{equation}
Indeed, the estimation \eqref{1p11} was obtained in \cite[(4.26)]{LY1} and the estimation \eqref{1p12} was obtained in \cite[(4.28)]{LY1}.\\
Comparing to our integral \eqref{1p1}-\eqref{1p11} and (\ref{1p2})-\eqref{1p12}, the differences are the
following integrals:%
\begin{equation*}\begin{aligned}
&\int_{\Omega}\Big\{ \frac{1}{{\varepsilon}^{2}}e^{u_0+\tilde{W}_{\bf{x},\mu}+\eta_{{\bf{x}},\mu}}
 (1-e^{u_0+\tilde{W}_{\bf{x},\mu}+\eta_{{\bf{x}},\mu}})
 \\& -\frac{1}{{\varepsilon}^{2}}e^{G(1+u_0+\tilde{W}_{\bf{x},\mu}+\eta_{{\bf{x}},\mu})}
 (1-e^{G(1+u_0+\tilde{W}_{\bf{x},\mu}+\eta_{{\bf{x}},\mu})})^2\Big\} Y_{x_{i},\mu_{i},j}dx,
\end{aligned}\end{equation*}
and%
\begin{equation*}\begin{aligned}
&\int_{\Omega}\Big\{ \frac{1}{{\varepsilon}^{2}}e^{u_0+\tilde{W}_{\bf{x},\mu}+\eta_{{\bf{x}},\mu}}
 (1-e^{u_0+\tilde{W}_{\bf{x},\mu}+\eta_{{\bf{x}},\mu}})
 \\& -\frac{1}{{\varepsilon}^{2}}e^{G(1+u_0+\tilde{W}_{\bf{x},\mu}+\eta_{{\bf{x}},\mu})}
 (1-e^{G(1+u_0+\tilde{W}_{\bf{x},\mu}+\eta_{{\bf{x}},\mu})})^2\Big\} Y_{{\bf{x}},\mu,0}dx.
\end{aligned}\end{equation*}
From \eqref{totaldiff}, we remind that
\begin{equation}\label{estforu21}
\begin{aligned}
&   {g}_{{\bf{x}},\mu}\left(\eta\right)   -\tilde{g}_{{\bf{x}},\mu}\left(\eta\right)
\\&=\frac{1}{{\varepsilon}^{2}}e^{u_0+\tilde{W}_{\bf{x},\mu}+\eta}
 (1-e^{u_0+\tilde{W}_{\bf{x},\mu}+\eta})
  -\frac{1}{{\varepsilon}^{2}}e^{G(1+u_0+\tilde{W}_{\bf{x},\mu}+\eta)}
 (1-e^{G(1+u_0+\tilde{W}_{\bf{x},\mu}+\eta)})^2
 \\&=O\Big( \frac{e^{3u_0+3\tilde{W}_{\bf{x},\mu}+3\eta}}{\e^2}\Big)\ \ \textrm{on}\ \ \Omega.
\end{aligned}\end{equation}
From Lemma \ref{tildeU}, we see that \begin{equation}\begin{aligned} e^{\tilde{W}_{\bf{x},\mu}}=O\Big(\frac{\e}{(1+\mu_i^2|y-x_i|^2)^2}\Big)\  \textrm{on}\  B_{2d_i}(x_i)\ \textrm{and}\ e^{\tilde{W}_{\bf{x},\mu}}=O(\e^3)\ \textrm{on}\ \Omega\setminus[\cup B_{d_i}(x_i)].
\label{estforu2}\end{aligned}\end{equation}
%We also remind that \begin{equation}\begin{aligned} \|\eta_{{\bf{x}},\mu}\|_{\LI}=O\Big(\frac{\ln\mu}{\mu^{2-\frac{\alpha}{2}}}\Big).\label{estforu2}\end{aligned}\end{equation}
Then from   \eqref{estforu21}-\eqref{estforu2}, we see  that
\begin{equation}\begin{aligned}\label{estforu3}
&\int_{\Omega}\Big\{ \frac{1}{{\varepsilon}^{2}}e^{u_0+\tilde{W}_{\bf{x},\mu}+\eta_{{\bf{x}},\mu}}
 (1-e^{u_0+\tilde{W}_{\bf{x},\mu}+\eta_{{\bf{x}},\mu}})
  \\&
  -\frac{1}{{\varepsilon}^{2}}e^{G(1+u_0+\tilde{W}_{\bf{x},\mu}+\eta_{{\bf{x}},\mu})}
 (1-e^{G(1+u_0+\tilde{W}_{\bf{x},\mu}+\eta_{{\bf{x}},\mu})})^2\Big\}  Y_{x_{i},\mu_{i},j}dy\\&
=\int_{\Omega}O\Big(\frac{e^{3u_0+3\tilde{W}_{\bf{x},\mu}}}{{\varepsilon}^{2}}\Big)Y_{x_{i},\mu_{i},j}dy =O\Big(\frac{1}{\mu^{3}}\Big).\end{aligned}\end{equation}
Similarly, we also see that from   \eqref{estforu21}-\eqref{estforu2},
\begin{equation}\begin{aligned}\label{estforu4}
&\int_{\Omega}\Big\{ \frac{1}{{\varepsilon}^{2}}e^{u_0+\tilde{W}_{\bf{x},\mu}+\eta_{{\bf{x}},\mu}}
 (1-e^{u_0+\tilde{W}_{\bf{x},\mu}+\eta_{{\bf{x}},\mu}})
  \\&-\frac{1}{{\varepsilon}^{2}}e^{G(1+u_0+\tilde{W}_{\bf{x},\mu}+\eta_{{\bf{x}},\mu})}
 (1-e^{G(1+u_0+\tilde{W}_{\bf{x},\mu}+\eta_{{\bf{x}},\mu})})^2\Big\}  Y_{{\bf{x}},\mu,0}dy
\\&
=\int_{\Omega}O\Big(\frac{e^{3u_0+3\tilde{W}_{\bf{x},\mu}}}{{\varepsilon}^{2}}\Big) Y_{{\bf{x}},\mu,0}dy
 =O\Big(\frac{1}{\mu^5}\Big).\end{aligned}\end{equation}
In view of \eqref{1p11}-\eqref{estforu4}, we complete the proof of Proposition \ref{pppp}.
\end{proof}

\textbf{Completion of the proof of Theorem \ref{blmix}}:
In view of Proposition \ref{pp}, we can find $\eta_{{\bf{x}},\mu}$ satisfying \eqref{final}. To complete the proof of  Theorem \ref{blmix}, we need to find  $({\bf{x}},\mu)$ suitably  depending on $\varepsilon>0$ such that the corresponding $c_{0}$, $c_{ij}$ are zero in \eqref{final}. By using  Proposition \ref{pppp},
we see that the conditions in Lemma \ref{ppp} are equivalent to
\begin{equation}DG^*({\bf{x}})=O(\frac{\ln\mu}{\mu^{2-\frac{\alpha}{2}}}),\label{final1}\end{equation}and
\begin{equation}\begin{aligned}\label{final2}&\frac{8}{\rho_{1}\mu^{3}}\left(
%TCIMACRO{\dsum \limits_{i=1}^{k}}%
%BeginExpansion
{\displaystyle \sum \limits_{i=1}^{k}}
%EndExpansion
\rho_{i}\left(  \int_{\Omega_{i}\setminus B_{d_i}\left(  x_i\right)  }%
\frac{e^{f_{{\bf{x}},i}}-1}{\left \vert y-x_{i}\right \vert ^{4}}-\int_{\mathbb{R}^2\setminus
\Omega_{i}}\frac{1}{\left \vert y-x_{i}\right \vert ^{4}}\right)  \right)
+ B_0{\varepsilon}^{2}\mu \\&\
=\frac{1}{\mu^{3}}O\left(  \left \vert DG^{\ast}\left(  {\bf{x}}\right)  \right \vert
^{2}\ln \mu+\delta^{2}\right)  +O\left(  \frac{1}{\mu^{5}}\right).\end{aligned}\end{equation}
Since $D({\bf{q}})<0$, we can find a small $\delta>0$, such that for $\bf{x}$ close to $\bf{q}$, we have
\begin{equation*}\begin{aligned}
%TCIMACRO{\dsum \limits_{i=1}^{k}}%
%BeginExpansion
{\displaystyle \sum \limits_{i=1}^{k}}
%EndExpansion
\rho_{i}\left(  \int_{\Omega_{i}\setminus B_{d_i}\left(  x_i\right)  }%
\frac{e^{f_{{\bf{x}},i}}-1}{\left \vert y-x_{i}\right \vert ^{4}}-\int_{\mathbb{R}^2\setminus
\Omega_{i}}\frac{1}{\left \vert y-x_{i}\right \vert ^{4}}\right) +O(\delta^2)
<0.\end{aligned}\end{equation*}
Then we obtain a solution $({\bf{x}},\mu )=({\bf{x}}(\e),\mu(\e))$ of \eqref{final1}-\eqref{final2} satisfying
\begin{equation*}|DG^*({\bf{x}}(\e))|=O(\e^{1-\frac{\alpha}{4}}\ln\e), \ \ \mu(\e)\in\Big(\frac{\beta_0}{\sqrt{\e}},\frac{\beta_1}{\sqrt{\e}}\Big),\end{equation*}
which implies the existence of a solution $u_\e$ to \eqref{1}.
In view of  $e^{\tilde{W}_{\bf{x},\mu}}=O(\e)$ on $\Omega$, $u_\e=F(v_\e)$,    and  \begin{equation*}
\begin{aligned}
&u_\e(y)=1+u_0+\tilde{W}_{\bf{x},\mu}+\eta_{{\bf{x}},\mu},
\end{aligned}
\end{equation*}
we obtain that  $\lim_{\e\to0} \sup_{\Omega}v_{\e}=-\infty$.
Moreover, we remind that from \cite{LY1},
  \begin{equation}\label{est2}\begin{aligned}\int_{B_{d_i}(x_i)}e^{w_{{\bf{x}},\mu}^*+u_0}dx=\frac{8^{k-1}\rho_1}{\Pi_{i=2}^k\mu_i^2}\Big(8\pi+O(\frac{\ln\mu}{\mu^2})\Big),\end{aligned}\end{equation} \begin{equation}\label{est3}\begin{aligned} \int_{\Omega}e^{w_{{\bf{x}},\mu}^*+u_0}dx=\frac{8^{k-1}\rho_1}{\Pi_{i=2}^k\mu_i^2}\Big(8k\pi+O(\frac{\ln\mu}{\mu^2})\Big),
\end{aligned}\end{equation}
and \begin{equation}\label{est4}\begin{aligned}w_{{\bf{x}},\mu}^*(x)=\sum_{i=1}^k w_{x_i,\mu_i}^*(x)=-2k\ln\mu+O(1)\ \textrm{ on}\  \Omega\setminus[\cup_{i=1}^{k}B_\delta(x_i))\  \textrm{for any}\  \delta>0.\end{aligned}\end{equation}
Indeed,   the estimation \eqref{est2} was obtained in \cite[(2.9)]{LY1},   the estimation \eqref{est3} was obtained in \cite[(2.10)]{LY1}, and
the estimation \eqref{est4} was obtained in \cite[(2.12)]{LY1}.
Then we obtain that
\[\frac{e^{G(1+u_0+\tilde{W}_{\bf{x},\mu}+\eta_{{\bf{x}},\mu})}}{\int_{\Omega}e^{G(1+u_0+\tilde{W}_{\bf{x},\mu}+\eta_{{\bf{x}},\mu})}dx}=\frac{e^{G(u_{\e})}}{\int_{\Omega}e^{G(u_{\e})}dx
}=
\frac{e^{v_{\e}}}{\int_{\Omega}e^{v_{\e}}dx
}\rightarrow \frac{1}{k}\sum_{i=1}^{k}\delta_{q_{i}},
\]in the sense of
measure as  $\e\to0$.
At this point, we complete the proof of Theorem \ref{blmix}.\hfill$\square$

{\bf Acknowledgement}\\
The author wishes to thank the anonymous referees very much for careful reading
and valuable comments.

\end{document}